\documentclass[11pt,reqno]{amsart}
\usepackage{amsmath,amsthm}

\newtheorem{theorem}{Theorem}[section]
\newtheorem{lemma}[theorem]{Lemma}
\newtheorem{proposition}[theorem]{Proposition}
\newtheorem{corollary}[theorem]{Corollary}

\theoremstyle{definition}

\newtheorem{definition}[theorem]{Definition}

\newtheorem{remark}[theorem]{Remark}
\newtheorem{notation}[theorem]{Notation}

\numberwithin{equation}{section}

\begin{document}

\title{Generalized de Sitter Space in $n$-dimensional Minkowski Space}

\author{David N. Pham}
\address{Department of Mathematics\\
Marymount Manhattan College\\
NY, NY 10021}
\email{dpham90@gmail.com,\hspace*{0.05in} dpham@mmm.edu}

\begin{abstract}
In this paper, we generalize the defining equation for de Sitter space by replacing the de Sitter radius with a function $f$ satisfying certain conditions; each resulting hypersurface is diffeomorphic to de Sitter space, and has a geometry (and causal character) which is controlled by the choice of $f$.  Necessary and sufficient conditions are obtained for a hypersurface to be timelike, null, or spacelike in the generalized model; in the non-null case, the geometry is given by a warped product.  Several examples of timelike, null, and spacelike hypersurfaces are presented.   Lastly, we calculate the Ricci tensor and scalar curvature for a special family of 4-dimensional generalized de Sitter spaces. 
\end{abstract}

\maketitle
\section{Introduction}
\noindent For $n\ge 1$, let $\mathbb{R}^{n+2}_1$ denote $(n+2)$-dimensional Minkowski space, that is, the space $\mathbb{R}^{n+2}$ with metric
\begin{equation}
\eta:=-dx^0\otimes dx^0+dx^1\otimes dx^1+\cdots+dx^{n+1}\otimes dx^{n+1}.
\end{equation}
$(n+1)$-dimensional de Sitter space $dS_{n+1}(r)$ (cf. \cite{HE} \cite{SSV}) is simply the $(n+1)$-dimensional hyperboloid 
\begin{equation}
\label{hyperboloid}
-(x^0)^2+(x^1)^2+\cdots + (x^{n+1})^2=r^2
\end{equation}
 in $\mathbb{R}^{n+2}_1$,  with the induced metric;  the constant $r>0$ in (\ref{hyperboloid}) is called the de Sitter radius.  It can be shown that $dS_{n+1}(r)$ is isometric to $\mathbb{R}\times S^n$ with metric
 \begin{equation}
 \label{deSitterMetric}
 -r^2~dt\otimes dt+r^2\cosh^2 t~ \Omega_n,
 \end{equation}
where $\Omega_n$ denotes the round metric on the unit $n$-sphere $S^n$.   It follows immediately from (\ref{deSitterMetric}) that $dS_{n+1}(r)$ is a time orientable Lorentzian manifold.  

From the point of view of cosmology, the aforementioned isometry allows $dS_{n+1}(r)$ to be interpreted as a spherical universe which contracts exponentially to a minimal radius $r$ before re-expanding at an exponential rate.  In addition to this, $dS_{n+1}(r)$ is also a vacuum solution of the Einstein field equations with positive cosmological constant, that is,
\begin{equation}
\label{Einstein}
\mbox{Ric}-\frac{1}{2}Sg +\Lambda g=0,
\end{equation}
where $g$ is the metric, $\mbox{Ric}$ is the Ricci tensor, $S$ is the scalar curvature, and $\Lambda$ is the cosmological constant, which, in the case of $dS_{n+1}(r)$, is 
\begin{equation}
 \Lambda=\left(\frac{n-1}{2(n+1)}\right)S=\frac{n(n-1)}{2r^2}.  
\end{equation}

In this paper, we generalize the defining equation for de Sitter space by replacing the de Sitter radius with a function $f:(a,b)\rightarrow \mathbb{R}$ satisfying certain conditions; each resulting hypersurface is diffeomorphic to de Sitter space, and has a geometry (and causal character) which is controlled by the choice of $f$.  Keeping in line with the notation for de Sitter space, we define $dS_{n+1}(f)$ to be the hypersurface given by
\begin{equation}
\label{hypersurfaceGeneral}
-(x^0)^2+(x^1)^2+\cdots + (x^{n+1})^2=f(x^0)^2,
\end{equation}
with the induced metric from the ambient Minkowski space.

The rest of the paper is organized as follows. In section 2, we recall some definitions and background from the theory of Lorentzian manifolds.  In section 3, we state and prove the main result of the paper.  Necessary and sufficient conditions are obtained for a hypersurface to be timelike, null, or spacelike in the generalized model; in the non-null case, the geometry is given by a warped product.  Several examples of timelike, null, and spacelike hypersurfaces are presented.  The flexibility of the generalized model is demonstrated by constructing a spacetime which expands for $t>0$, but has a constant, non-zero radius for $t\le 0$.  Lastly, in section 4, we calculate the Ricci tensor and scalar curvature for a special family of 4-dimensional generalized de Sitter spaces; the results are then compared with that of ordinary de Sitter space.

\section{Preliminaries} 
\subsection{Lorentzian Manifolds}
A Lorenztian manifold is a pair $(M,g)$ where $M$ is a smooth manifold of dimension $n\ge 2$ and $g$ is a symmetric, nondegenerate tensor field of type (0,2) with signature $(-+\cdots +)$.  A tangent vector $v\in T_pM$ can be classified in to three classes according to the sign of $g(v,v)$.  Specifically, $v$ is
\begin{itemize}
\item[(i)] timelike if $g(v,v)<0$
\item[(ii)] null if $g(v,v)=0$ and $v\neq 0$
\item[(iii)] spacelike if $g(v,v)>0$.
\end{itemize}
$v$ is called causal if $v$ is timelike or null.  If $u\in T_pM$ is timelike, then it can be shown that $g(u,v)\neq 0$ for all causal vectors $v$.  This fact allows one to define an equivalence relation $\sim$ on the set of timelike vectors of $T_pM$ by taking $u\sim v$ if $g(u,v)<0$.  Moreover, if $n\in T_pM$ is null and $u,v\in T_pM$ are timelike then $g(u,n)$ and $g(v,n)$ are of the same sign iff $u\sim v$.  Consequently, fixing an equivalence class of timelike vectors on $T_pM$ defines a \textit{time orientation} on the set of causal vectors of $T_pM$.  Specifically, if $u\in T_pM$ is a representative of the chosen equivalence class and $v\in T_pM$ is causal, then $v$ belongs to the \textit{future causal cone} (\textit{past causal cone}) if $g(u,v)<0$ ($g(u,v)>0$).  

A Lorentzian manifold $(M,g)$ is time orientable if it admits a smooth, non-vanishing timelike vector field; $(M,g)$ is then time oriented if, for every $p\in M$, $T_pM$ is smoothly assigned a time orientation, that is, the future causal cone of $T_pM$ for all $p\in M$ is fixed by some timelike vector field.  Its straightforward to show that if $M$ is connected and time orientable, then there are only two possible time orientations. 
\begin{definition}
A \textit{spacetime} is a connected, time oriented (thus time orientable) Lorentzian manifold.
\end{definition}
\noindent A hypersurface $\Sigma$ of a Lorentzian manifold $(M,g)$ is called 
\begin{itemize}
\item[(i)] timelike if the normal to $T_p\Sigma$ for all $p\in \Sigma$ is spacelike
\item[(ii)] null if the normal to $T_p\Sigma$ for all $p\in \Sigma$ is null
\item[(iii)] spacelike if the normal to $T_p\Sigma$ for all $p\in \Sigma$ is timelike.
\end{itemize} 
\noindent In other words, if $i: \Sigma\hookrightarrow M$ is the inclusion map, then $(\Sigma,i^\ast g)$ is Lorentizan (Riemannian) iff $\Sigma$ is timelike (spacelike); in addition, $\Sigma$ is null iff $i^\ast g$ is degenerate.

\subsection{Warped Products}
In this section, we recall the definition of the warped product of pseduo-Riemannian manifolds (cf. \cite{BE} \cite{On}).
\begin{definition}
\label{WarpedProductDef}
Let $(B,g_B)$ and $(F,g_F)$ be pseduo-Riemannian manifolds and let $k>0$ be a smooth function on $B$.  Then the warped product $B\times_k F$ is the product manifold $B\times F$ with metric
\begin{equation}
g:=\pi^\ast_Bg_B+(k\circ\pi_B)^2\pi^\ast_F g_F
\end{equation}
where $\pi_B$ and $\pi_F$ are the natural projections of $B\times F$ onto $B$ and $F$, respectively.
\end{definition}
\noindent In Definition \ref{WarpedProductDef}, the function $k$ is called the \textit{warping function}.


\section{The Hypersurface $dS_{n+1}(f)$}
\begin{theorem}
\label{deSitterGeneral}
If $\alpha:\mathbb{R}\rightarrow \mathbb{R}$ is a smooth function for which $\alpha>0$ for all $t$ and
\begin{equation}
h(t):=\alpha(t)\sinh t
\end{equation}
has nonzero derivative for all $t$, then $dS_{n+1}(f)$ with $f=\alpha\circ h^{-1}$ is isometric to $\mathbb{R}\times S^n$ with metric
\begin{equation}
\label{generalMetric}
g=((\alpha')^2-\alpha^2)dt\otimes dt+\alpha^2 \cosh^2 t~\pi^\ast\Omega_n
\end{equation}
where $\alpha':=\frac{d\alpha}{dt}$, $\Omega_n$ is the round metric on $S^n$, and $\pi: \mathbb{R}\times S^n\rightarrow S^n$ is the projection map.  
\end{theorem}
\begin{proof}
Let $(a,b)\subset\mathbb{R}$ denote the domain of $f$ and note that this is simply the image of $h$ (which must be an open set since $h$ is a diffeomorphsim onto its image).  We now show that  $dS_{n+1}(f)$ is a smooth hypersurface of $(a,b)\times \mathbb{R}^{n+1}\subset \mathbb{R}^{n+2}$.  To see this, let $F:(a,b)\times \mathbb{R}^{n+1}\rightarrow \mathbb{R}$ be the map defined by
\begin{equation}
F(x)= -(x^0)^2+(x^1)^2+\cdots +(x^{n+1})^2-f(x^0)^2.
\end{equation}
Then $dS_{n+1}(f)=F^{-1}(0)$ and 
\begin{equation}
dF = 2\left(-(x^0+f'f)~dx^0+\sum_{i=1}^{n+1}x^i ~dx^i\right).
\end{equation}
Since $f\neq 0$, it follows that $dF_x\neq 0$ for all $x\in F^{-1}(0)$.  This implies that every element of $F^{-1}(0)$ is a regular point which proves that $F^{-1}(0)$ is a smooth hypersurface of $(a,b)\times \mathbb{R}^{n+1}$ (and hence a smooth hypersurface of  $\mathbb{R}^{n+2}$) (cf. Corollary 8.10 of \cite{L}).

Next, we verify that $dS_{n+1}(f)$ is diffeomorphic to  $\mathbb{R}\times S^n$.  To do this, let $\varphi: \mathbb{R}^{n+2}\rightarrow \mathbb{R}^{n+2}$ be the smooth map defined by
\begin{align}
x^0&=\alpha(t)\sinh t\\
x^i &= z^i\alpha(t)\cosh t,~1\le i\le n+1
\end{align}
where $t:=z^0$.  Since $h$ is a diffeomorphism onto its image, it follows that $\varphi:\mathbb{R}^{n+2}\rightarrow \varphi(\mathbb{R}^{n+2})$ is invertible with inverse given by 
\begin{align}
\label{tDef}
t&=h^{-1}(x^0)\\
z^i&=\frac{x^i}{\alpha(h^{-1}(x^0))\cosh (h^{-1}(x^0))},~1\le i\le n+1.
\end{align}
Since $h:\mathbb{R}\rightarrow (a,b)$ is a diffeomorphism and $\alpha$ is nonzero and smooth, it follows that the above inverse map is smooth.  This shows that $\varphi:  \mathbb{R}^{n+2}\rightarrow \varphi(\mathbb{R}^{n+2})=(a,b)\times \mathbb{R}^{n+1}$ is a diffeomorphism.  

Setting
\begin{equation}
\omega(x):=-(x^0)^2+(x^1)^2+\cdots +(x^{n+1})^2,
\end{equation}
a direct calculation shows that a point $(t,z)\in \mathbb{R}^{n+2}$ satisfies
\begin{equation}
\omega(\varphi(t,z))=\alpha^2(t)
\end{equation}
iff $z\in S^n$.   This shows that $\varphi(\mathbb{R}\times S^n)=dS_{n+1}(f)$.  This fact along with the fact that $\varphi$ is a diffeomorphism (onto its image) and $\mathbb{R}\times S^n\subset \mathbb{R}^{n+2}$  and $dS_{n+1}(f)\subset \varphi(\mathbb{R}^{n+2})$ are embedded submanifolds proves that 
\begin{equation}
\label{dSn(f)Diffeo}
\varphi|_{\mathbb{R}\times S^n}:\mathbb{R}\times S^n\rightarrow  dS_{n+1}(f)
\end{equation} 
is a diffeomorphism.

For (\ref{dSn(f)Diffeo}) to be an isometry, the metric on $\mathbb{R}\times S^n$ must necessarily be
\begin{equation}
\label{generalMetric1}
g=(\varphi|_{\mathbb{R}\times S^n})^\ast i^\ast\eta=j^\ast\varphi^\ast\eta,
\end{equation}
where $i: dS_{n+1}(f)\hookrightarrow \mathbb{R}^{n+2}_1$ and $j: \mathbb{R}\times S^n\hookrightarrow\mathbb{R}^{n+2}$ are the inclusion maps.  To compute (\ref{generalMetric1}),  note that 
\begin{align}
\nonumber
-j^\ast\varphi^\ast(dx^0\otimes dx^0)&=-(\alpha' \sinh t+\alpha\cosh t)^2 dt\otimes dt\\
\label{x0}
&=-\left((\alpha')^2\sinh^2 t+2\alpha\alpha' \sinh t\cosh t+\alpha^2\cosh^2 t\right)dt\otimes dt
\end{align}
and 
\begin{align}
\nonumber
j^\ast\varphi^\ast(dx^i\otimes dx^i) &= (\alpha'\cosh t+\alpha \sinh t)^2(z^i\circ j)^2 dt\otimes dt+\alpha^2 \cosh^2 t~j^\ast(dz^i\otimes dz^i)\\
\label{xi}
&+(z^i\circ j)(\alpha \cosh t)(\alpha'\cosh t+\alpha \sinh t)(dt\otimes j^\ast(dz^i)+j^\ast(dz^i)\otimes dt)
\end{align}
for $1\le i\le n+1$.   Since 
\begin{equation}
\sum_{i=1}^{n+1} (z^i\circ j)^2=1,
\end{equation}
it follows that 
\begin{equation}
\label{zn+1}
(z^{n+1}\circ j)j^\ast(dz^{n+1})=-\sum_{i=1}^n (z^i\circ j) j^\ast(dz^i).
\end{equation}
Substituting (\ref{zn+1}) into the last term of (\ref{xi}) for $i=n+1$ gives
\begin{align}
\nonumber
j^\ast\varphi^\ast(dx^{n+1}\otimes &dx^{n+1})=(\alpha'\cosh t+\alpha \sinh t)^2(z^{n+1}\circ j)^2 dt\otimes dt+\alpha^2 \cosh^2 t~j^\ast(dz^{n+1}\otimes dz^{n+1})\\
\nonumber
&-\sum_{i=1}^n(z^i\circ j)(\alpha \cosh t)(\alpha'\cosh t+\alpha \sinh t)(dt\otimes j^\ast(dz^i)+j^\ast(dz^i)\otimes dt).
\end{align}
It follows from this that
\begin{align}
\nonumber
\sum_{i=1}^{n+1} j^\ast\varphi^\ast(dx^i\otimes dx^i)&= (\alpha'\cosh t+\alpha \sinh t)^2 dt\otimes dt+\alpha^2 \cosh^2 t\sum_{i=1}^{n+1}j^\ast(dz^i\otimes dz^i)\\
\nonumber
&= \left((\alpha')^2\cosh^2 t+2\alpha\alpha'\cosh t \sinh t+\alpha^2\sinh^2t\right) dt\otimes dt\\
\label{xiSum}
&+\alpha^2 \cosh^2 t\sum_{i=1}^{n+1}j^\ast(dz^i\otimes dz^i).
\end{align}
Using (\ref{x0}) and (\ref{xiSum}), we have
\begin{align}
\nonumber
g&=-j^\ast\varphi^\ast(dx^0\otimes dx^0)+\sum_{i=1}^{n+1}j^\ast\varphi^\ast(dx^i\otimes dx^i)\\
\nonumber
&=((\alpha')^2-\alpha^2)dt\otimes dt+\alpha^2 \cosh^2 t\sum_{i=1}^{n+1}j^\ast(dz^i\otimes dz^i)\\
&=((\alpha')^2-\alpha^2)dt\otimes dt+\alpha^2 \cosh^2 t~\pi^\ast\Omega_n.
\end{align}  
\end{proof}
\begin{notation}
For the remainder of this paper, we will adopt a slight abuse of notation and denote $\pi^\ast\Omega_n$ as $\Omega_n$.
\end{notation}

\noindent Motivated by Theorem \ref{deSitterGeneral}, we now introduce the following definitions:
\begin{definition}
Let $\Psi$ denote the set of smooth functions $\alpha:\mathbb{R}\rightarrow \mathbb{R}$ which satisfy the following two conditons:
\begin{itemize}
\item[(i)] $\alpha>0$ for all $t$, and
\item[(ii)] $h(t):=\alpha(t)\sinh t$ has non-zero derivative for all $t$.
\end{itemize}
\end{definition}
\begin{definition}
Let $\hat{\Psi}$ be the set of all $\alpha\in\Psi$ which satisfy $|\alpha'|\neq \alpha$ for all $t$.  
\end{definition}
\begin{definition}
For $\alpha\in \Psi$, let $f_\alpha$ be the smooth function defined by
\begin{equation}
\label{fDef}
f_\alpha:=\alpha\circ h^{-1}
\end{equation}
where $h(t):=\alpha(t)\sinh t$.
\end{definition}
\begin{remark}
As a special case, we note that when $\alpha\equiv r$ for $r>0$ a constant, we have $dS_{n+1}(f_\alpha)=dS_{n+1}(r)$.  As one would expect from this, the metric in Theorem \ref{deSitterGeneral} does indeed coincide with the de Sitter metric (see (\ref{deSitterMetric})) when $\alpha\equiv r$.
\end{remark}
The following are some immediate consequences of Theorem  \ref{deSitterGeneral}.
\begin{corollary}
If $\alpha\in \hat{\Psi}$, then $dS_{n+1}(f_\alpha)$ is isometric to a warped product of $(\mathbb{R},\beta dt\otimes dt)$ and $(S^n,\Omega_n)$ where 
\begin{itemize}
\item[1.] $\beta:=(\alpha')^2-\alpha^2$, and
\item[2.] the warping function is $\alpha\cosh t$.
\end{itemize}
\end{corollary}
\begin{corollary}
Let $\alpha\in \Psi$.  Then $dS_{n+1}(f_\alpha)$ is
\begin{itemize}
\item[(i)] a timelike hypersurface iff $\alpha>|\alpha'|$ for all $t$
\item[(ii)] a null hypersurface iff $\alpha=re^{\epsilon t}$ where $\epsilon =\pm 1$ and $r>0$
\item[(iii)] a spacelike hypersurface iff $\alpha<|\alpha'|$ for all $t$.
\end{itemize}
\end{corollary}
\begin{proof}
(i) and (iii) follow immediately from Theorem \ref{deSitterGeneral}.  For $(ii)$, its straightforward to show that $re^{\epsilon t}\in \Psi$ for $\epsilon=\pm 1$ and $r>0$.  From Theorem  \ref{deSitterGeneral},  $dS_{n+1}(f_\alpha)$ is null precisely when
\begin{equation}
\label{preciseNull}
(\alpha')^2-\alpha^2=0~\forall t.
\end{equation}
(\ref{preciseNull}) is equivalent to the condition that $\alpha'=\pm \alpha$.  The uniqueness of linear ODE's implies that $\alpha=re^{\epsilon t}$ where $\epsilon$ is $1$ or $-1$.  In addition, $r>0$ since $\alpha\in \Psi$.  
\end{proof}

\begin{corollary}
\label{HypersurfaceFamily}
Let $\alpha=re^{\lambda t}$ where $r>0$ and $\lambda$ are constants.  Then
\begin{itemize}
\item[(i)] $\alpha\in\Psi$ iff $-1\le \lambda \le 1$ 
\item[(ii)] $dS_{n+1}(f_\alpha)$ is a timelike hypersurface iff $-1<\lambda < 1$
\item[(iii)] $dS_{n+1}(f_\alpha)$ is a null hypersurface iff $\lambda = 1$ or $-1$.
\end{itemize}
\end{corollary}
\begin{proof}
(i).  Suppose that there exists some $t_0\in\mathbb{R}$ for which 
\begin{equation}
\label{PsiCondition0}
\frac{d}{dt} re^{\lambda t} \sinh t|_{t=t_0} =0.
\end{equation}
Its straightforward to show that (\ref{PsiCondition0}) is equivalent to the condition that 
\begin{equation}
\label{PsiCondition1}
e^{2t_0}(\lambda+1)=\lambda -1.
\end{equation}
It follows immediately from this that (\ref{PsiCondition1}) has a solution $t_0\in \mathbb{R}$ precisely when $|\lambda|>1$.  Hence, $\alpha \in  \Psi$ iff $|\lambda|\le 1$.  

(ii) and (iii).  This follow immediately from Theorem \ref{deSitterGeneral} and the fact that 
\begin{equation}
(\alpha')^2-\alpha^2=\alpha^2(\lambda^2-1).
\end{equation} 
\end{proof}
\begin{corollary}
\label{deSitterS}
Let $\alpha=r\cosh t$.  Then $\alpha\in \hat{\Psi}$ and $dS_{n+1}(f_\alpha)$ is a timelike hypersurface isometric to $\mathbb{R}\times S^n$ with metric
\begin{equation}
\label{deSitterS1}
g=-r^2 dt\otimes dt+r^2\cosh^4 t\Omega_n.
\end{equation}
\end{corollary}
\begin{corollary}
\label{deSitterT}
Let $\alpha(t)=\frac{r}{\cosh t}$ where $r>0$ is a constant.  Then $\alpha\in \hat{\Psi}$ and $dS_{n+1}(f_\alpha)$ is a timelike hypersurface isometric to $\mathbb{R}\times S^n$ with metric
\begin{equation}
\label{deSitterT1}
g=-\frac{r^2}{\cosh^4 t} ~dt\otimes dt+r^2\Omega_n.
\end{equation}
\end{corollary}
Next, we show that $\Psi$ also contains $\alpha$ for which $dS_{n+1}(f_\alpha)$ is spacelike.  We begin with the following lemma:
\begin{lemma}
\label{SpacelikeConditionLemma}
If $\alpha:\mathbb{R}\rightarrow \mathbb{R}$ is a smooth positive function which satisfies 
\begin{itemize}
\item[(i)] $\frac{\alpha'}{\alpha}\tanh t>-1$ for all $t$
\item[(ii)] $\frac{\alpha'}{\alpha}>1$ for all $t$,
\end{itemize}
then $\alpha\in\Psi$ and $dS_{n+1}(f_\alpha)$ is spacelike.
\end{lemma}
\begin{proof}
Statement (i) is equivalent to the condition that 
\begin{equation}
\frac{d}{dt}(\alpha\sinh t)=\alpha'\sinh t+\alpha\cosh t>0
\end{equation}
for all $t$, which proves that $\alpha\in \Psi$.  (ii) implies that 
\begin{equation}
(\alpha')^2-\alpha^2>0
\end{equation}
for all $t$, which in turn implies that $dS_{n+1}(f_\alpha)$ is spacelike.
\end{proof}
\begin{proposition}
\label{SpacelikeAlphaProp}
There exists an $\alpha\in \Psi$ for which $dS_{n+1}(f_\alpha)$ is spacelike.
\end{proposition}
\begin{proof}
Let  $\zeta=1+\rho$ where $\rho$ is some smooth positive function on $\mathbb{R}$ that is yet to be determined.  We will chose $\rho$ to satisfy the condition $\zeta \tanh t>-1$ for all $t$.  With $\rho>0$ for all $t$, it is straightforward to show that the aforementioned condition is equivalent to the statement that 
\begin{equation}
\label{conditionRho}
\rho<\frac{2e^t}{e^{-t}-e^t}\hspace*{0.1in}\forall t<0.
\end{equation}    
To satisfy (\ref{conditionRho}), we set 
\begin{equation}
\rho:=\frac{2e^t}{e^{-t}+e^t}.
\end{equation}
With this choice of $\rho$, the inequalities $\zeta\tanh t>-1$ and $\zeta>1$ are satisfied for all $t$.  By setting 
\begin{equation}
\alpha(t):=e^{\int_0^t \zeta~ds},
\end{equation}
we see that $\frac{\alpha'}{\alpha}=\zeta$.  By Lemma \ref{SpacelikeConditionLemma}, $\alpha\in \Psi$ and $dS_{n+1}(f_\alpha)$ is spacelike.
\end{proof}
\noindent The $\alpha$ constructed in the proof of Proposition \ref{SpacelikeAlphaProp} satisfies the condition
\begin{equation}
\lim_{t\rightarrow -\infty}\frac{\alpha}{\alpha'}=1.
\end{equation}
This is far from a coincidence as the next result shows.
\begin{proposition}
Suppose $\alpha\in \Psi$ and $dS_{n+1}(f_\alpha)$ is a spacelike hypersurface.  Then one of the following must be true:
\begin{itemize}
\item[(i)] $\displaystyle \lim_{t\rightarrow -\infty}\frac{\alpha}{\alpha'}=1$, or
\item[(ii)] $\displaystyle \lim_{t\rightarrow \infty}\frac{\alpha}{\alpha'}=-1$.
\end{itemize}
\end{proposition}
\begin{proof}
Since $dS_{n+1}(f_\alpha)$ is spacelike, it follows from Theorem \ref{deSitterGeneral} that
\begin{equation}
\label{derivativeCondition0}
|\alpha'|>\alpha~~\forall~t,
\end{equation}
which in turn implies that $\alpha'\neq 0$ for all $t$.   Since $\alpha\in \Psi$, we have
\begin{equation}
\frac{d}{dt}\alpha(t)\sinh t\neq 0~~\forall ~t,
\end{equation}
which in turn is equivalent to 
\begin{equation}
\label{derivativeCondition1}
\alpha'\tanh t+\alpha\neq 0~~\forall ~ t.
\end{equation}
Since $\mathbb{R}$ is connected and (\ref{derivativeCondition1}) evaluated at $t=0$ is $\alpha(0)>0$, we have
\begin{equation}
\label{derivativeCondition2}
\alpha'\tanh t+\alpha>0~~\forall~ t.
\end{equation}
Combining (\ref{derivativeCondition2}) with (\ref{derivativeCondition0}) gives
\begin{equation}
\label{derivativeCondition3}
-\alpha'\tanh t <\alpha<|\alpha'|.
\end{equation}
Since $\mathbb{R}$ is connected and $\alpha'\neq 0~\forall~t$, it follows that $\alpha'>0~\forall~t$ or $\alpha'<0$ for all $t$.  If $\alpha'>0~\forall~t$, then (\ref{derivativeCondition3}) implies that 
\begin{equation}
\label{derivativeCondition4}
-\tanh t<\frac{\alpha}{\alpha'}<1.
\end{equation}
Taking the limit of (\ref{derivativeCondition4}) as $t\rightarrow -\infty$ gives (i).   If $\alpha'<0$ for all $t$,  then (\ref{derivativeCondition3}) implies that
\begin{equation}
\label{derivativeCondition5}
-1<\frac{\alpha}{\alpha'}<-\tanh t.
\end{equation}
Taking the limit of (\ref{derivativeCondition5}) as $t\rightarrow \infty$ gives (ii).
\end{proof}

So far, we have shown that the generalized model presented in this paper contains timelike, null, and spacelike hypersurfaces.   In the case of null hypersurfaces, the precise form of the defining equation is known.   
In Corollary \ref{deSitterS}, we gave an example of a spacetime which contracts for $t<0$ and expands for $t>0$, which is what ordinary de Sitter space does.  In Corollary \ref{deSitterT}, we gave an example of a spacetime whose radius remains constant for all time.  We conclude this section by showing that the generalized model contains enough flexibility to merge characteristics from both of these spacetimes to produce a spacetime which expands for $t>0$, but has a constant, nonzero radius for $t\le 0$.
\begin{proposition}
\label{NoContraction}
There exists an $\alpha\in \hat{\Psi}$ so that $dS_{n+1}(f_\alpha)$ is isometric to $\mathbb{R}\times S^n$ with metric
\begin{equation}
g=\beta(t) ~dt\otimes dt+a(t)^2~\Omega_n
\end{equation}
where $a,~\beta:\mathbb{R}\rightarrow \mathbb{R}$ are smooth functions for which
\begin{itemize}
\item[(i)] $a(t)$ is a positive constant for $t\le 0$ and an increasing function for $t>0$
\item[(ii)] $\displaystyle\lim_{t\rightarrow \infty} a(t)=\infty$
\item[(iii)] $\beta(t)<0$ $\forall t$
\item[(iv)] $\displaystyle\lim_{t\rightarrow \infty} \beta(t)=-1$.
\end{itemize}
\end{proposition}
\begin{proof}
We begin with a proof of (iii). Let 
\begin{equation}
\label{alphaDef}
\alpha(t):=\frac{a(t)}{\cosh t}
\end{equation}
where $a(t)>0$ is a smooth, non-decreasing function which will be specified shortly and let
\begin{equation}
\label{omegaDef}
\beta(t):= (\alpha')^2-\alpha^2.
\end{equation}
A direct calculation shows that $\beta(t)<0$ iff
\begin{equation}
\label{timelikeCondition}
a^2(\tanh^2 t-1)-2aa'\tanh t +(a')^2<0.
\end{equation}
Since $a^2(\tanh^2 t-1)<0$, the condition that
\begin{equation}
\label{timelikeCondition1}
a'(a'-2a\tanh t)\le 0
\end{equation}
is sufficient to ensure that $\beta < 0$ for all $t$.  To define $a$, let $\mu$ be the smooth function\footnote{For a proof that $\mu$ is smooth, see Lemma 2.20 of \cite{L}.} given by
\begin{equation}
\mu:=\left\{\begin{array}{ll}
e^{-1/t^2} & t>0\\
0 & t\le 0
\end{array}.\right.
\end{equation}
Then we define $a(t)$ to be
\begin{equation}
\label{aDef}
a(t):= \mu(t) \cosh t+r
\end{equation}
where $r>0$ is a constant which will be determined shortly.   Since $a\equiv r$ for $t\leq 0$, it follows that $a'\equiv 0$ for $t<0$.  Since $a'$ is continuous (actually smooth), we also have $a'(0)=0$.  From this, we see that (\ref{timelikeCondition1}) is satisfied for $t\le 0$.  Since $a'>0$ for $t>0$, it suffices to show that for $t>0$,
\begin{equation}
\label{timelikeCondition2}
a'-2a\tanh t\le 0,
\end{equation}
which in turn, is equivalent to the condition that
\begin{equation}
\label{timelikeCondition3}
\mu'(t)-\mu(t)\tanh t-\frac{2r \tanh t}{\cosh t}\le 0
\end{equation}
for $t>0$.  Note that $\mu(t)\tanh(t)$ is an increasing function for $t>0$ for which 
\begin{equation}
\lim_{t\rightarrow \infty} \mu(t)\tanh(t) =1,
\end{equation}
while $\mu'$ is increasing for $0<t<\sqrt{2/3}$ and decreasing for $t>\sqrt{2/3}$.   Since $\mu'(\sqrt{2/3})<1$, there exists a $t_0>0$ such that 
\begin{equation}
\label{inequality0}
\mu(t)\tanh t> \mu'(\sqrt{2/3})
\end{equation}
for $t\ge t_0$; for concreteness, we set $t_0=2.4$.  The function 
\begin{equation}
\rho(t):=\frac{2 \tanh t}{\cosh t}
\end{equation}
is increasing for $0<t<\sinh^{-1}(1)\simeq 0.88$ and decreasing for  $t>\sinh^{-1}(1)$.  Since
\begin{align}
\mu'(0)=\rho(0)=0
\end{align}
and
\begin{align}
\mu''(0)=0,~~\rho'(0)>0,
\end{align}
it follows that there exists an $\varepsilon>0$ arbitrarily small so that 
\begin{equation}
\mu'(t)<\rho(t):=\frac{2 \tanh t}{\cosh t}
\end{equation}
for $0<t<\varepsilon$.  If $\varepsilon>\sinh^{-1}(1)$, replace $\varepsilon$ with any positive number less than $\sinh^{-1}(1)$.  Now choose $r_0>1$ so that the following inequalities are both satisfied:
\begin{align}
\frac{2r_0 \tanh \varepsilon/2}{\cosh \varepsilon/2} & > \mu'(\sqrt{2/3})\\
\frac{2r_0 \tanh t_0}{\cosh t_0} & > \mu'(\sqrt{2/3}).
\end{align}
Since $\rho(t)$ is increasing for $0<t<\sinh^{-1}(1)$ and decreasing for  $t>\sinh^{-1}(1)$, it folows that  
\begin{equation}
\label{inequality1}
\frac{2r \tanh t}{\cosh t}>\mu'(t)
\end{equation}
for $0<t\le t_0$ and $r\ge r_0$.   Fixing $r\ge r_0$ and combining (\ref{inequality1}) with (\ref{inequality0}) proves (\ref{timelikeCondition3}) for $t>0$; this in turn completes the proof of statement (iii) of Proposition \ref{NoContraction}. 

With $a(t)$ given by (\ref{aDef}), we have $a'=0$ for $t\le 0$ and $a'>0$ for $t>0$; this proves statement (i) of Proposition \ref{NoContraction}.  Statement (ii) is immediate from the definition of $a(t)$.  For statement (iv), we have 
\begin{align}
\lim_{t\rightarrow \infty}& \beta(t)=\lim_{t\rightarrow \infty} (\alpha')^2-\alpha^2\\
\nonumber
&=\lim_{t\rightarrow \infty}(\mu+\frac{r}{\cosh t})^2(\tanh^2 t-1)-2(\mu+\frac{r}{\cosh t})(\mu'+\mu\tanh t)\tanh t\\
\nonumber
&+(\mu'+\mu\tanh t)^2 \\
\nonumber
&=-1.
\end{align}

Lastly, let 
\begin{equation}
h(t)=\alpha(t)\sinh(t).
\end{equation}
Its straightforward to show that $h'(t)\neq 0$ for all $t$.  Theorem \ref{deSitterGeneral} implies that $dS_{n+1}(f_\alpha)$ is isometric to $\mathbb{R}\times S^n$ with metric\begin{equation}
\beta(t)~dt\otimes dt+a(t)^2~\Omega_n.
\end{equation}
\end{proof}

\section{Calculation of $\mbox{Ric}$ and $S$ for $dS_4(f_\alpha)$ $(\alpha=re^{\lambda t})$}
We conclude this note by observing how the Ricci tensor and scalar curvature of the generalized model differ from ordinary de Sitter space for the case of $dS_4(f_\alpha)$, where $\alpha=re^{\lambda t}$ with $r$ and $\lambda$ constants satisfying $r>0$ and $-1<\lambda<1$.   (Note that $dS_4(r)$ arises as a special case of $dS_4(f_\alpha)$ upon setting $\lambda=0$.)   The choice of $\alpha=re^{\lambda t}$ is of special interest since all of the null hypersurfaces in the generalized model are obtained by setting $\lambda = \pm 1$.  

For ordinary $n+1$-dimensional de Sitter space $dS_{n+1}(r)$, the Ricci tensor and scalar curvature are given by
\begin{align}
&R_{\mu\nu}=\frac{n}{r^2} g_{\mu\nu}\\
&S=\frac{n(n+1)}{r^2}
\end{align}
where $R_{\mu\nu}:=\mbox{Ric}_{\mu\nu}$ and $g$ is the de Sitter metric.  

By Corollary \ref{HypersurfaceFamily}, $dS_4(f_\alpha)$ is a timelike hypersurface precisely when $-1<\lambda<1$.  Also, by Theorem \ref{deSitterGeneral}, $dS_4(f_\alpha)$ is isometric to $\mathbb{R}\times S^3$ with metric 
\begin{equation}
\nonumber 
g=((\alpha')^2-\alpha^2)dt\otimes dt+\alpha^2\cosh^2 t~\Omega_3.
\end{equation}
Using hyperspherical coordinates $(\psi,\theta,\phi)$ for $S^3\subset \mathbb{R}^4$ where
\begin{align}
x^1&=\cos\psi\\
x^2&=\sin\psi\cos\theta\\
x^3&=\sin\psi\sin\theta\cos\phi\\
x^4&=\sin\psi\sin \theta \sin \phi,
\end{align}
the non-zero components of the metric $g$ of $dS_4(f_\alpha)$  (with local coordinates $(t,\psi,\theta,\phi)$ for $\mathbb{R}\times S^3$) are
\begin{align}
g_{tt}&=\alpha^2(\lambda^2-1)\\
g_{\psi\psi} &=\alpha^2\cosh^2 t\\
g_{\theta\theta}&=\alpha^2\cosh^2 t\sin^2\psi\\
g_{\phi\phi}&=\alpha^2\cosh^2 t \sin^2 \psi \sin^2\theta.
\end{align}
\noindent  
For the Christoffel symbols\footnote{The Einstein summation convention of summing over repeated upper and lower indices is in effect in (\ref{Christoffel}) and for the rest of the paper.}
\begin{equation}
\label{Christoffel}
\Gamma^\rho_{\mu\nu} = \frac{1}{2}g^{\rho\sigma}(\partial_\mu g_{\nu\sigma} +\partial_\nu g_{\mu\sigma}-\partial_\sigma g_{\mu\nu}),
\end{equation}
define 
\begin{equation}
q_\lambda(t):=\frac{(\lambda+\tanh t)\cosh^2 t}{1-\lambda^2}.
\end{equation}
Then the nonzero\footnote{The complete list of nonzero Christoffel symbols can be obtained from the above list by using the fact that $\Gamma^\rho_{\mu\nu}=\Gamma^\rho_{\nu\mu}$.} Christoffel symbols are
\begin{align}
&\Gamma^t_{tt}=\lambda,\hspace*{0.05in}\Gamma^t_{\psi\psi}=q_\lambda(t),\hspace*{0.05in}\Gamma^{t}_{\theta\theta}=q_\lambda(t)\sin^2\psi,\hspace*{0.05in} \Gamma^t_{\phi\phi}=q_\lambda(t)\sin^2\psi\sin^2\theta\\
&\Gamma^\psi_{t\psi}=\lambda+\tanh t,\hspace*{0.05in} \Gamma^\psi_{\theta\theta}=-\sin\psi\cos \psi,\hspace*{0.05in}\Gamma^\psi_{\phi\phi}=-\sin^2\theta \sin \psi\cos \psi\\
&\Gamma^\theta_{t\theta}=\lambda+\tanh t,\hspace*{0.05in} \Gamma^\theta_{\phi\phi}=-\sin\theta \cos \theta,\hspace*{0.05in} \Gamma^\theta_{\psi\theta}=\cot \psi\\
&\Gamma^\phi_{t\phi}=\lambda+\tanh t,\hspace*{0.05in}\Gamma^\phi_{\psi\phi} =\cot \psi,\hspace*{0.05in} \Gamma^\phi_{\theta\phi} =\cot \theta. 
\end{align}
For the Ricci tensor
\begin{equation}
R_{\mu\nu}:=R^\rho_{~\mu\rho\nu}=\partial_\rho\Gamma^\rho_{\nu\mu}-\partial_\nu\Gamma^\rho_{\rho\mu}+\Gamma^\rho_{\rho\lambda}\Gamma^\lambda_{\nu\mu}-\Gamma^\rho_{\nu\lambda}\Gamma^\lambda_{\rho\mu},
\end{equation}
define
\begin{equation}
f_\lambda(t) := \frac{(\lambda\cosh t+ \sinh t)(3\sinh t+2\lambda \cosh t)+3-2\lambda^2}{1-\lambda^2}.
\end{equation}
Then the non-zero components of the Ricci tensor are
\begin{align}
\label{Rt}
&R_{tt}=-3(1+\lambda\tanh t)\\
\label{Rpsi}
&R_{\psi\psi} = f_\lambda(t)\\
\label{Rtheta}
&R_{\theta \theta}= f_\lambda(t)\sin^2\psi \\
\label{Rphi}
&R_{\phi\phi}=f_\lambda(t)\sin^2\psi\sin^2\theta.
\end{align}
The scalar curvature 
\begin{equation}
S:=g^{\mu\nu}R_{\mu\nu}
\end{equation}
is then 
\begin{equation}
\label{Scalc}
S=\frac{3}{\alpha^2}\left(\frac{1+\lambda\tanh t}{1-\lambda^2}+\frac{f_\lambda(t)}{\cosh^2 t}\right).
\end{equation}
Setting $\lambda=0$, equations (\ref{Rt})-(\ref{Rphi}) and (\ref{Scalc}) reduce to 
\begin{align}
& R_{tt}=-3\\
& R_{\psi\psi}= 3\cosh^2 t\\
& R_{\theta\theta}=3\sin^2\psi \cosh^2t\\
& R_{\phi\phi}=3\sin^2 \psi\sin^2\theta \cosh^2 t\\
& S= \frac{12}{r^2},
\end{align}
 which (as expected) are the Ricci tensor and scalar curvature of $dS_4(r)$.    So we see from equation (\ref{Scalc}) that any deviation from $\lambda= 0$ results in a scalar curvature for $dS_4(f_\alpha)$  that varies with $t$.  In particular, 
 \begin{align}
 \lim_{t\rightarrow \infty} S =0,\hspace*{0.1in} \lim_{t\rightarrow -\infty} S=\infty
 \end{align}
 for $ 0<\lambda<1$ and
 \begin{align}
 \lim_{t\rightarrow \infty} S =\infty,\hspace*{0.1in}\lim_{t\rightarrow -\infty} S = 0
 \end{align}
  for $-1<\lambda < 0$.

\end{document}